\documentclass[11pt,reqno]{amsart}
\usepackage[all]{xy}
\usepackage{amssymb}
\usepackage{amsthm}
\usepackage{amsmath,mathtools}
\usepackage{amscd,enumitem}
\usepackage{verbatim}
\usepackage{eurosym}
\usepackage{float}
\usepackage{color}
\usepackage{dcolumn}
\usepackage[mathscr]{eucal}
\usepackage[all]{xy}
\usepackage{bbm}
\usepackage[textheight=8.5in, textwidth=6.7in]{geometry}
\newtheorem*{conj*}{Conjecture}
\newtheorem{theorem}{Theorem}[section]

\theoremstyle{definition}
\newtheorem*{remark}{Remark}
\theoremstyle{plain}

\newtheorem{lemma}[theorem]{Lemma}

\newtheorem{prop}[theorem]{Proposition}

\newtheoremstyle{named}{}{}{\itshape}{}{\bfseries}{.}{.5em}{\thmnote{#3 }#1}
\theoremstyle{named}
\newtheorem*{namedconjecture}{Conjecture}

\newcommand{\R}{\mathbb{R}}

\newcommand{\C}{\mathbb{C}}

\DeclareMathOperator\Log{Log}
\numberwithin{equation}{section}

\newtheoremstyle{example}
  {\topsep}   
  {\topsep}   
  {\normalfont}  
  {0pt}       
  {\bfseries} 
  {.}         
  {5pt plus 1pt minus 1pt} 
  {}          
\theoremstyle{example}

\def\({\left(}
\def\){\right)}

\def\lp{\left(}
\def\rp{\right)}

\usepackage{centernot}

\title{Seaweed Algebras and the Index Statistic for Partitions}
\author{William Craig}
\address{Department of Mathematics\\University of Virginia\\ Kerchof Hall 112\\ 141 Cabell Drive \\ Charlottesville, VA 22903}
\email{wlc3vf@virginia.edu}

\keywords{Index of partitions, non-negative coefficients, seaweed algebras, circle method}
\subjclass[2020]{05A15, 11P82}

\begin{document}

\maketitle

\begin{abstract}
In 2018 Coll, Mayers, and Mayers conjectured that the $q$-series $\lp q, -q^3; q^4 \rp_\infty^{-1}$ is the generating function for a certain parity statistic related to the index of seaweed algebras. We prove this conjecture. Thanks to earlier work by Seo and Yee, the conjecture would follow from the non-negativity of the coefficients of this infinite product. Using a variant of the circle method along with Euler--Maclaurin summation, we establish this non-negativity, thereby confirming the Coll--Mayers--Mayers Conjecture.
\end{abstract}

\section{Introduction and Statement of Results}

Recall that a partition $\lambda$ of the non-negative integer $n$ is a non-increasing sequence of positive integers $\lambda = \lp \lambda_1 \geq \lambda_2 \geq \dots \geq \lambda_k \rp$ which sum to $n$. We adopt the usual convention that the empty set is the only partition of zero. Define the $q$-series $G(q)$ by
\begin{align*}
    G(q) := \sum_{n \geq 0} a(n) q^n := \dfrac{1}{\lp q, -q^3; q^4 \rp_\infty},
\end{align*}
where $\lp a; q \rp_\infty := \prod_{n=0}^\infty \lp 1 - a q^n \rp$ and $\lp a, b; q \rp_\infty := \lp a; q \rp_\infty \lp b; q \rp_\infty$. The coefficients $a(n)$ can be found in the Online Encyclopedia of Integer Sequences \cite{OEIS} as entry A300574. Interestingly, the values $a(n)$ appear to always be non-negative. Coll, Mayers, and Mayers in \cite{CollMayersMayers} have conjectured a combinatorial interpretation for the sequence $a(n)$, from which non-negativity would follow.

The conjectured combinatorial interpretation of $a(n)$ in \cite{CollMayersMayers} is expressed in terms of the {\it index} of a partition $\lambda$, which is defined in terms of certain Lie algebras called {\it seaweed algebras}. Let $\{ e_j \}_{1 \leq j \leq n}$ be the standard basis of $k^n$ for some field $k$. Given two partitions $\{ a_j \}_{1 \leq j \leq m}$, $\{ b_j \}_{1 \leq j \leq \ell}$ of $n$, Dergachev and Kirillov \cite{DergachevKirillov} defined seaweed algebras as Lie subalgebras of $\text{Mat}(n)$ which preserve the vector spaces $\text{span}\lp e_1, e_2, \dots, e_{a_1 + \dots + a_j} \rp$ for $1 \leq j \leq m$ and $\text{span}\lp e_{b_1 + \dots + b_j + 1}, \dots, e_n \rp$ for $1 \leq j \leq \ell$. 

In \cite[Theorem 5.1]{DergachevKirillov}, Dergachev and Kirillov obtain an exact formula for the index of seaweed algebras. Motivated by the parameterization of seaweed algebras by partitions, Coll, Mayers, and Mayers define the index of a pair of partitions $\lp \lambda, \mu \rp$ of $n$ as the index of the associated seaweed algebra. There are a number of interesting specializations of the index. For example, certain indexes are related to the 2-colored partition function \cite{CollMayersMayers}.

The specialization which we call the {\it index} of the partition $\lambda \vdash n$ is the index of the pair $\lp \lambda, \{ n \} \rp$. Studying this index statistic, Coll, Mayers, and Mayers defined $e_n$ as the number of partitions of $n$ into odd parts whose index is even, and likewise $o_n$ as the number of partitions of $n$ into odd parts whose index is odd. With this notation, the conjecture of Coll, Mayers, and Mayers may be stated as follows.

\begin{namedconjecture}[Coll--Mayers--Mayers] \label{MAIN CONJECTURE}
The following are true:

\noindent \textnormal{(1)} All the coefficients of $G(q)$ are non-negative.

\noindent \textnormal{(2)} We have $G(q) = \sum\limits_{n \geq 0} \left| e_n - o_n \right| q^n$.
\end{namedconjecture}

In recent years, several papers have made progress towards proving this conjecture. Seo and Yee \cite{SeoYee} proved that (1) implies the full Coll--Mayers--Mayers Conjecture\footnote{Seo and Yee also conjectured non-negativity for the coefficients of $\lp q, -q^{m-1}; q^m \rp_\infty^{-1}$ for all $m \geq 4$, and if non-negativity is assumed then \cite[Corollary 3]{SeoYee} would yield combinatorial interpretations similar to the Coll--Mayers--Mayers Conjecture for $m = 4d$.}. After this work, Chern \cite{Chern} demonstrated using the circle method that $a(n) \geq 0$ for all $n > 2.4 \times 10^{14}$ and verified the non-negativity by computer for $0 \leq n \leq 10000$. In this paper, we use a modified approach, involving a different version of the circle method and Euler-Maclaurin summation, which reduces the last possible counterexample to $n \leq 350000$ and subsequently proves the Coll--Mayers--Mayers Conjecture.

\begin{theorem} \label{MAIN}
The Coll--Mayers--Mayers Conjecture is true.
\end{theorem}

In the process of proving this result, we prove an effective asymptotic for the value of $a(n)$, which we state here for completeness.

\begin{theorem} \label{a(n) Asymptotics}
As $n \to \infty$, we have
\begin{align*}
	a(n) \sim \dfrac{\Gamma\lp \frac 14 \rp \pi^{\frac 14}}{2^{\frac 94} 3^{\frac 38} n^{\frac 38}} I_{-\frac 34} \lp \dfrac{\pi}{2} \sqrt{\dfrac{n}{3}} \rp + (-1)^n \dfrac{\Gamma\lp \frac 34 \rp \pi^{\frac 34}}{2^{\frac{11}{4}} 3^{\frac 58} n^{\frac 58}} I_{-\frac 54}\lp \dfrac{\pi}{2} \sqrt{\dfrac{n}{3}} \rp,
\end{align*}
and for $n > 4800$ the difference between these has absolute value at most $E(n)$ as defined in \eqref{Error Definition}.
\end{theorem}

\begin{remark}
Theorem \ref{a(n) Asymptotics} reproduces the asymptotic main term in \cite[Theorem 1.2]{Chern}.
\end{remark}

The remainder of the paper is structured as follows. In Section \ref{Preliminaries} we lay out preliminary facts needed in the proof of Theorem \ref{MAIN}. In Section \ref{Estimates of Key Terms}, we collect various effective estimates related to $G(q)$. Finally, in Section \ref{Proof of MAIN} we prove Theorems \ref{MAIN} and \ref{a(n) Asymptotics} by a variation of Wright's circle method.

\section*{Acknowledgements}
The author thanks Ken Ono, his Ph.D advisor, and Shane Chern for helpful discussions related to the results in this paper. The author also thanks Ken Ono, Kathrin Bringmann, Shane Chern, and Josh Males for helpful comments on earlier versions of this paper, and the anonymous referees for pointing out an error in the original manuscript and for many helpful comments that improved the exposition of the paper. The author thanks the support of Ken Ono's grants, namely the Thomas Jefferson Fund and the NSF (DMS-1601306 and DMS-2055118).

\section{Preliminaries} \label{Preliminaries}

\subsection{Bernoulli Polynomials} \label{Bernoulli Section}

We first recall important facts about the {\it Bernoulli polynomials} $B_n(x)$. These polynomials are defined classically by their generating function
\begin{align*}
    \sum_{n \geq 0} \dfrac{B_n(x)}{n!} z^n = \dfrac{z e^{(x-1)z}}{1 - e^{-z}}.
\end{align*}
The {\it Bernoulli numbers} are the constant terms of these polynomials, i.e. $B_n = B_n(0)$. The Bernoulli polynomials appear prominently throughout number theory and satisfy many interesting and important properties. In our application, what will matter is that the Bernoulli polynomials appear within coefficients in certain infinite series closely related to $G(q)$, and we need to know their size in the interval $[0,1]$ in order to bound these coefficients. In particular, Lehmer \cite{Lehmer} proved\footnote{This is an application of the Fourier expansion for the Bernoulli polynomials.} for $n \geq 2$ and $0 \leq x \leq 1$ the bound
\begin{align} \label{Lehmer's Bound}
    \left| B_n(x) \right| \leq \dfrac{2 \zeta(n) n!}{(2\pi)^n}.
\end{align}
We will also make use of the infinite series
\begin{align*}
    B_{r,t}(z) := \dfrac{e^{-\frac rt z}}{z\lp 1 - e^{-z} \rp} = \sum_{n \geq -2} \dfrac{B_{n+2}\lp 1 - \frac rt \rp}{(n+2)!} z^n,
\end{align*}
where $0 < r \leq t$ are integers. Due to Lehmer's bound, this Laurent expansion is absolutely convergent in the punctured disk $0 < |z| < 2\pi$. This absolute convergence is important for producing effective estimates of certain infinite sums related to $G(q)$, which will be seen in Lemma \ref{B_rt Bound}.

\subsection{Euler--Maclaurin Approximation}

For integers $M, N > 0$ and any analytic function $f(z)$, the classical Euler-Maclaurin summation formula says
\begin{align*}
    \int_0^M f(z) dz = \dfrac{f(0) + f(M)}{2} + \sum_{m = 1}^{M-1} f(m) + \sum_{n = 1}^{N-1} \dfrac{(-1)^{n+1} B_{n+1}}{(n+1)!} &\lp f^{(n)}(M) - f^{(n)}(0) \rp \\ &+ (-1)^N \int_0^M f^{(N)}(x) \dfrac{\widehat{B}_N(x)}{N!} dx,
\end{align*}
where $\widehat{B}_n(x) := B_n\lp x - \lfloor x \rfloor \rp$. Zagier observed \cite[Formula (44)]{Zagier} that for functions $f(z)$ with suitable growth conditions at infinity, the Euler-Maclaurin formula provides a very precise tool for approximating certain infinite sums involving $f(z)$. A modest generalization of Zagier's observation is that if $f(z) \sim \sum_{n = 0}^\infty c_n z^n$ as $z \to 0$ in a  conical region $D_\delta := \{ z : 0 < \left| \mathrm{Arg}{z} \right| < \frac{\pi}{2} - \delta \}$, and if $f(z)$ and its derivatives decay faster than any negative power of $z$ as $z \to \infty$, then
\begin{align*}
    \sum_{m \geq 0} f\lp (m+a)z \rp \sim \dfrac{I_f}{z} - \sum_{n = 0}^\infty c_n \dfrac{B_{n+1}(a)}{n+1} z^n
\end{align*}
as $z \to 0$ in $D_\delta$ for any real number $0 < a \leq 1$, where $I_f := \int_0^\infty f(x) dx$. The growth condition imposed on $f(z)$ near infinity is commonly referred to as $f(z)$ having {\it rapid decay} at infinity in the literature. The symbol $\sim$ for asymptotic expansions is used in a strong sense, namely we say $f(z) \sim \sum_{n=0}^\infty c_n z^n$ if for all integers $N > 0$ we have $f(z) - \sum_{n=0}^{N-1} c_n z^n = O\lp z^N \rp$ as $z \to 0$. Because the proof of this expansion is based on an exact formula, it may be readily refined in a variety of ways. One of the most desirable generalizations in practice is to allow $f(z)$ to decay more slowly towards infinity, requiring only that $f(z) = O\lp z^{- 1 - \epsilon} \rp$ for some $\epsilon > 0$ as $z \to \infty$, in which case we say $f(z)$ has {\it sufficient decay} at infinity. Such functions may have asymptotic expansions of the form $f(z) \sim \sum_{n = n_0}^\infty c_n z^n$ as $z \to 0$ in $D_\delta$ for any integer $n_0$. In this case, one may essentially treat the principal parts and non-principal parts separately, and this treatment gives a more general lemma.

To state this lemma we require some notation. Let $\zeta(s,a) := \sum_{n \geq 1} \lp n + a \rp^{-s}$ be the {\it Hurwitz zeta function}, $\gamma$ the {\it Euler-Mascheroni constant}, and $\psi(a) := \frac{\Gamma^\prime(a)}{\Gamma(a)}$ the {\it digamma function}.

\begin{lemma}[{\cite[Lemma 2.2]{BCMO}}] \label{Euler-Maclaurin Sufficient Decay}
Let $0 < a \leq 1$ and $A, \delta \in \R^+$, and assume that $f(z)$ has the asymptotic expansion $f(z) \sim \sum_{n=n_0}^{\infty} c_n z^n$ as $z \rightarrow 0$ in $D_\delta$ for some integer $n_0$, possibly negative. Furthermore, assume that $f$ and all of its derivatives are of sufficient decay in $D_\delta$ as $z \to \infty$. Then
\begin{align*}
	\sum_{n=0}^\infty f((n+a)z)\sim \sum_{n=n_0}^{-2} c_{n} \zeta(-n,a)z^{n}+ \frac{I_{f,A}^*}{z}-\frac{c_{-1}}{z} \left( \Log \left(Az \right) +\psi(a)+\gamma \right)-\sum_{n=0}^\infty c_n \frac{B_{n+1}(a)}{n+1} z^n
\end{align*}
uniformly as $z \rightarrow 0$ in $D_\delta$, where 
\begin{align*}
	I_{f,A}^*:= \int_{0}^{\infty} \left(f(u)-\sum_{n=n_0}^{-2}c_{n}u^n-\frac{c_{-1}e^{-Au}}{u}\right)du.
\end{align*}
\end{lemma}

Since the classical Euler--Maclaurin formula is an exact formula, the error terms in these asymptotic expansions can be made completely explicit. In particular, when $f(z) = \sum_{n = n_0}^\infty c_n z^n$ in some punctured disk $0 < |z| < R$ has suitable decay conditions, these bounds have been explicitly computed. Define for any $N > 0$ the constants $M_N := \max\limits_{0 \leq x \leq 1} \left| B_N(x) \right|$ and
\begin{align*}
    J_{g,N}(z) := \int_0^\infty \left| g^{(N)}(w) \right| |dw|
\end{align*}
whenever $g$ is $N$-times differentiable, where the path of integration is along the line through $0$ and $z$.

\begin{lemma}[{\cite[Proposition 3.4]{Craig}}] \label{Euler-Maclaurin General Effective}
Let $f(z)$ be $C^\infty$ in $D_\delta$ with Laurent series $f(z) = \sum_{n = n_0}^\infty c_n z^n$ that converges absolutely in the region $0 < |z| < R$ for some positive constant $R$. Suppose $f(z)$ and all its derivatives have sufficient decay as $z \to \infty$ in $D_\delta$. Then for any real numbers $0 < a \leq 1$, $A > 0$ and any integer $N > 0$, we have
\begin{align*}
	\bigg| \sum_{m \geq 0} f\lp (m+a)z \rp - \sum_{n = n_0}^{-2} c_n \zeta(-n,a) z^n &- \dfrac{I_{f,A}^*}{z} + \dfrac{c_{-1}}{z} \lp \Log\lp Az \rp + \gamma + \psi\lp a \rp \rp + \sum_{n \geq 0} c_n^* \dfrac{B_{n+1}(a)}{n+1} z^n \bigg| \\ &\leq \dfrac{M_{N+1} J_{g,N+1}(z)}{(N+1)!} |z|^N + \sum_{k \geq N} |b_k| \lp 1 + \dfrac{k!}{10 (k-N)!} \rp |z|^k,
\end{align*}
where $g(z) := f(z) - \frac{c_{-1} e^{-Az}}{z} - \sum_{n = n_0}^{-2} c_n z^n$, $b_n := c_n - \frac{(-A)^{n+1} c_{-1}}{(n+1)!}$, $M_N$ and $J_{g,N}(z)$ are defined as above, and
\begin{align*}
	c_n^* := \begin{cases} c_n & \text{if } n \leq N-1, \\ \dfrac{(-A)^{n+1} c_{-1}}{(n+1)!} & \text{if } n \geq N. \end{cases}
\end{align*}
\end{lemma}

Because $B_{r,t}(z)$ is of sufficient decay and has a Laurent series converging for $0 < |z| < 2\pi$, Lemma \ref{Euler-Maclaurin General Effective} can be applied to $B_{r,t}(z)$ for $0 < |z| < 2\pi$. To state this application, we first introduce convenient notation. Let
\begin{align*}
    \beta_{r,t} := \log\left(\Gamma\left(\frac rt\right) \right) - \frac 12\log(2\pi), \hspace{0.5in} g_{r,t}(z) := B_{r,t}(z) - \frac{1}{z^2} - \frac{\lp \frac 12 - \frac rt \rp e^{-\frac rt z}}{z},
\end{align*}
and introduce the functions $F_a^{r,t}(z)$, $E_a^{r,t}(z)$ defined by
\begin{align} \label{B_rt Approx Def}
    F_a^{r,t}(z) := \dfrac{\zeta(2,a)}{z^2} + \dfrac{\beta_{r,t}}{z} - \dfrac{1}{z} \lp \frac 12 - \frac rt \rp \lp \Log(z) + \gamma + \psi(a) \rp + \sum_{n = 0}^\infty c_n^* \dfrac{B_{n+1}(a)}{n+1} z^n
\end{align}
and
\begin{align} \label{B_rt Error Def}
    E^{r,t}(z) := \dfrac{J_{g_{r,t},4}(z)}{720} |z|^3 + \sum_{k \geq 3} \left| \dfrac{B_{k+2}\lp 1 - \frac rt \rp}{(k+2)!} - \dfrac{\lp -r \rp^{k+1} \lp \frac 12 - \frac rt \rp}{t^{k+1} (k+1)!} \right| \lp 1 + \dfrac{k!}{10(k-3)!} \rp |z|^k,
\end{align}
where we define the coefficients $c_n^*$ as in Lemma \ref{Euler-Maclaurin General Effective} by
\begin{align*}
    c_n^* := \begin{cases} \dfrac{B_{n+1}\lp 1 - \frac rt \rp}{(n+2)!} & \text{ if } n \leq 2, \\ \dfrac{(-r)^{n+1} \lp \frac 12 - \frac rt \rp}{t^{n+1} (n+1)!} & \text{ otherwise}. \end{cases}
\end{align*}
We now state our application of Lemma \ref{Euler-Maclaurin General Effective} to $B_{r,t}(z)$.

\begin{lemma} \label{B_rt Bound}
Let $0 < r \leq t$ be integers and $\delta > 0$ a constant. Then for any real number $0 < a \leq 1$ and $z \in D_\delta$ with $0 < |z| < 2\pi$ , we have
\begin{align*}
    \bigg| \sum_{m \geq 0} B_{r,t}\lp (m+a)z \rp - F_a^{r,t}(z) \bigg| \leq E^{r,t}(z).
\end{align*}
\end{lemma}

\begin{proof}
$B_{r,t}(z)$ satisfies the criteria of Lemma \ref{Euler-Maclaurin General Effective}, and therefore for any $A > 0$ and $N = 3$ we have
\begin{align*}
    \bigg| \sum_{m \geq 0} B_{r,t}\lp (m+a)z \rp - \dfrac{\zeta(2,a)}{z^2} - \dfrac{I_{B_{r,t},A}^*}{z} + \dfrac{c_{-1}}{z}&\lp \Log(Az) + \gamma + \psi(a) \rp - \sum_{n = 0}^\infty c_n^* \dfrac{B_{n+1}(a)}{n+1} z^n \bigg| \\ &\leq \dfrac{M_4 J_{g_{r,t},4}}{24} |z|^3 + \sum_{k \geq 3} |b_k| \lp 1 + \dfrac{k!}{10(k-3)!} \rp |z|^k,
\end{align*}
where $b_k = \frac{B_{k+2}\lp 1 - \frac rt \rp}{(k+2)!} - \frac{\lp -r \rp^{k+1} \lp \frac 12 - \frac rt \rp}{t^{k+1} (k+1)!}$. To simplify the integral
\begin{align*}
    I_{B_{r,t},A}^* = \int_{0}^\infty \left(\dfrac{e^{-\frac rt z}}{z\lp 1 - e^{-z} \rp} - \dfrac{1}{z^2} + \lp \dfrac{r}{t} - \dfrac{1}{2} \rp \frac{e^{-Az}}{z} \right) dz,
\end{align*}
we use the substitutions $z \mapsto \frac{t}{r} z$ and $A = \frac{r}{t}$, which gives
\begin{align*}
    I_{B_{r,t},\frac{r}{t}}^* = \int_{0}^\infty \left(\dfrac{e^{-z}}{z \lp 1 - e^{- \frac{t}{r} z} \rp} - \dfrac{1}{\frac{t}{r} z^2} + \lp \dfrac{r}{t} - \dfrac{1}{2} \rp \frac{e^{-z}}{z} \right) dz.
\end{align*}
\cite[Lemma 2.3]{BCMO} states that for any real number $N > 0$,
\begin{multline*}
\int_0^\infty\left(\frac{e^{-x}}{x\left(1-e^{Nx}\right)}-\frac{1}{Nx^2}+\left(\frac 1N-\frac 12\right)\frac{e^{-x}}{x} \right)dx
=\log\left(\Gamma\left(\frac 1N\right) \right) +\left(\frac 12-\frac 1N\right) \log\left(\frac 1N\right)-\frac 12\log(2\pi),
\end{multline*}
and so the case $N = \frac tr$ implies
\begin{align*}
    I_{B_{r,t},\frac rt}^* = \log\left(\Gamma\left(\frac rt\right) \right) +\left(\frac 12-\frac rt\right) \log\left(\frac rt\right)-\frac 12\log(2\pi) = \beta_{r,t} + \left(\frac 12-\frac rt\right) \log\left(\frac rt\right).
\end{align*}
A short calculation therefore shows
\begin{align*}
    \bigg| \sum_{m \geq 0} B_{r,t}\lp (m+a)z \rp - \dfrac{\zeta(2,a)}{z^2} - \dfrac{I^*_{B_{r,t},\frac rt}}{z} - \dfrac{1}{z}\lp \dfrac{1}{2} - \dfrac{r}{t} \rp & \lp \Log\lp\frac{r}{t} z\rp + \gamma + \psi(a) \rp - \sum_{n = 0}^\infty c_n^* \dfrac{B_{n+1}(a)}{n+1} z^n \bigg| \\ &\leq \dfrac{J_{g_{r,t},4}}{720} |z|^3 + \sum_{k \geq 3} |b_k| \lp 1 + \dfrac{k!}{10(k-3)!} \rp |z|^k.
\end{align*}
By the definitions \eqref{B_rt Approx Def} and \eqref{B_rt Error Def}, this completes the proof.
\end{proof}

\section{Estimates of Key Terms} \label{Estimates of Key Terms}

The proof of Theorem \ref{MAIN} uses a variation of Wright's circle method. As with any variation of the circle method, there are various stages where estimates must be made. This section collects together the most important estimates, which are subdivided into three groups. The first two are dedicated to proving bounds on $G(q)$ on the {\it major arc} and {\it minor arc}, which play central roles in Wright's circle method and are defined in the first part. The last part considers elementary bounds on the functions $F_a^{r,t}(z)$ and $E_a^{r,t}(z)$ which make later computations more straightforward.

\subsection{Effective Major Arc Bounds}

Before we proceed, we define the terms {\it major arc} and {\it minor arc}. When using Wright's circle method, one must define the {\it major arc}, which is the region of some circle $C$ with fixed radius $|q|$, where $q = e^{-z}$ lies near a dominant pole of the generating function. In most examples, the dominant pole lies near $q=1$ and only one major arc is required. In our case, however, we will require two major arcs, which lie near $q = \pm 1$. The major arc near $q=1$ will consist of those $q = e^{-z}$ for which $z = x + iy$ satisfies $0 \leq |y| < 15x$, and the corresponding constraint near $q = -1$ is $\pi - 15x < |y| \leq \pi$. We will in practice use a change of coordinates $q \mapsto -q$ to translate the $q = -1$ major arc into the $q=1$ major arc of the function $G(-q)$, which gives back the restriction $0 \leq |y| < 15x$. The minor arc will consist of the complement of the two major arcs, that is, it consists of all $q = e^{-z}$ with $15x \leq |y| \leq \pi - 15x$. We begin now by deriving important bounds that hold on major arcs.

\begin{prop} \label{Major Arc Bound}
Let $q = e^{-z}$, $z = x + iy$ satisfy $x > 0$ and $0 \leq |y| < 15x$.

\noindent \textnormal{(1)} We have for $0 < x < \frac{2}{5t}$ that
\begin{align*}
    \left| \Log\lp\lp q^r; q^t \rp_\infty^{-1}\rp - tz F_1^{r,t}(tz) \right| \leq |tz| E^{r,t}(tz).
\end{align*}

\noindent \textnormal{(2)} We have for $0 < x < \frac{1}{5t}$ that
\begin{align*}
    \left| \Log\lp\lp -q^r; q^t \rp_\infty^{-1}\rp - tz F_1^{r,t}(2tz) + tz F_{1/2}^{r,t}(2tz) \right| \leq 2 |tz| E^{r,t}(2tz).
\end{align*}
\end{prop}

\begin{proof}
By expanding logarithms into Taylor series, we obtain
\begin{align} \label{Log Product Expansion}
    \Log\lp\lp \varepsilon q^r; q^t \rp_\infty^{-1}\rp = - \sum_{n \geq 0} \Log\lp 1 - \varepsilon q^{tn+r} \rp = \sum_{n \geq 0} \sum_{m \geq 1} \dfrac{\varepsilon^m q^{m(tn+r)}}{m} = \sum_{m \geq 1} \dfrac{\varepsilon^m q^{rm}}{m\lp 1 - q^{tm} \rp}.
\end{align}
Setting $q = e^{-z}$ and multiplying the above expression by $\frac{tz}{tz}$, we obtain
\begin{align*}
    \Log\lp\lp \varepsilon q^r; q^t \rp_\infty^{-1}\rp = tz \sum_{m \geq 1} \varepsilon^m \dfrac{e^{-rmz}}{tmz\lp 1 - e^{-tmz} \rp} = tz \sum_{m \geq 1} \varepsilon^m B_{r,t}(tmz),
\end{align*}
where $B_{r,t}(z)$ is defined as in Section \ref{Bernoulli Section}. Now, for $0 < x < \frac{2}{5}$ we have since $y^2 < 225x^2$ that $|z| = \sqrt{x^2 + y^2} < \frac{2\sqrt{226}}{5} < 2\pi$. Therefore, the Laurent expansion for $B_{r,t}(tz)$ is convergent for $0 < x < \frac{2}{5t}$, and likewise for $B_{r,t}(2tz)$ if $0 < x < \frac{1}{5t}$. If we set $\varepsilon = 1$, (1) follows directly from Lemma \ref{B_rt Bound}. If $\varepsilon = -1$, by applying Lemma \ref{B_rt Bound} to each summand of
\begin{align*}
    \Log\lp\lp -q^r; q^t \rp_\infty^{-1}\rp = tz \sum_{m \geq 0} B_{r,t}\lp (m+1) 2tz \rp - tz \sum_{m \geq 0} B_{r,t} \lp \lp m + \frac 12 \rp 2tz \rp,
\end{align*}
(2) follows as well.
\end{proof}

\subsection{Effective Minor Arc Bounds}

We now estimate $G(q)$ on the minor arc $15x \leq |y| \leq \pi - 15x$ when $x$ is small. In order to do this, we first prove several helpful results so that the proof of the main bound will be more readable.

\begin{lemma} \label{Alpha Estimations}
Let $m \geq 1$ be an integer, and $q = e^{-z}$, $z = x + iy$ with $0 < x < \frac{\pi}{480}$ and $15x \leq |y| < \frac{\pi}{2m}$. Then there exists a constant $\alpha_m > 0$ such that
\begin{align*}
    \dfrac{|q|^m}{m\left| 1 + \lp -1 \rp^{m+1} q^{2m} \right|} - \dfrac{|q|^m}{m\lp 1 - |q|^{2m} \rp} < \dfrac{e^{- \frac{m\pi}{480}}}{2m^2 x} \lp \dfrac{2m}{\alpha_m} - 1 \rp.
\end{align*}
Furthermore, in the cases $1 \leq m \leq 3$ we may choose $\alpha_1 = 29$, $\alpha_2 = 55$, and $\alpha_3 = 77$.
\end{lemma}

\begin{proof}
Since $15x \leq |y| < \frac{\pi}{2m}$, we have $\cos\lp 2my \rp \geq - \cos\lp 30mx \rp$, and so
\begin{align*}
    \left| 1 + \lp -1 \rp^{m+1} q^{2m} \right|^2 &= 1 - 2 \lp -1 \rp^{m+1} \cos\lp 2my \rp e^{-2mx} + e^{-4mx} \\ &\geq 1 - 2 \cos\lp 30mx \rp e^{-2mx} + e^{-4mx}.
\end{align*}
From the Taylor expansion
\begin{align*}
    1 - 2 \lp -1 \rp^{m+1} \cos\lp 30mx \rp e^{-2mx} + e^{-4mx} = 904 m^2 x^2 - 1808 m^3 x^3 + \cdots,
\end{align*}
it is apparent that $1 - 2 \lp -1 \rp^{m+1} \cos\lp 30mx \rp e^{-2mx} + e^{-4mx} > \alpha_m^2 x^2$ for some $\alpha_m > 0$ and $0 < x < \frac{\pi}{480}$. This shows that $\left| 1 + \lp -1 \rp^{m+1} q^{2m} \right| > \alpha_m x$ for all $0 < x < \frac{\pi}{480}$, and so
\begin{align*}
    \dfrac{|q|^m}{m\left| 1 + \lp -1 \rp^{m+1} q^{2m} \right|} - \dfrac{|q|^m}{m\lp 1 - |q|^{2m} \rp} < \dfrac{|q|^m}{m \alpha_m x} - \dfrac{|q|^m}{m\lp 1 - |q|^{2m} \rp}.
\end{align*}
By the inequalities $1 - |q|^{2m} = 1 - e^{-2mx} > 2mx$ and $|q|^m > e^{- \frac{m\pi}{480}}$ for $0 < x < \frac{\pi}{480}$, we arrive at the desired bound.

We now wish to evaluate $\alpha_1, \alpha_2$, and $\alpha_3$. Let $f_m(x) := 1 - 2 \lp -1 \rp^{m+1} \cos\lp 30mx \rp e^{-2mx} + e^{-4mx}$, and consider the auxiliary function $g_m(x) := f_m(x) - \alpha_m^2 x^2$. Note that $g_m(0) = g_m^\prime(0) = 0$ since both $f_m(x)$ and $x^2$ have a double zero at $x=0$. In order to prove that $f_m(x) > \alpha_m^2 x^2$ for $0 < x < \frac{\pi}{480}$, it will therefore suffice to prove that $g_m^{\prime\prime}(0) > 0$, i.e. that $f_m^{\prime\prime}(x) > 2 \alpha_m^2$, for $0 < x < \frac{\pi}{480}$. Now,
\begin{align*}
    f^{\prime\prime}_m(x) = 16 m^2 e^{-4 m x} \lp 1 + 112 e^{2 m x} \cos\lp 30 m x \rp - 15 e^{2 m x} \sin\lp 30 m x \rp \rp,
\end{align*}
and so the $\alpha_m$ we choose must satisfy
\begin{align*}
    \alpha_m^2 < 8 m^2 e^{-4 m x} \lp 1 + 112 e^{2 m x} \cos\lp 30 m x \rp - 15 e^{2 m x} \sin\lp 30 m x \rp \rp
\end{align*}
for all $0 < x < \frac{\pi}{480}$. For each $1 \leq m \leq 3$, $f^{\prime\prime}_m(x)$ is decreasing on the interval $0 < x < \frac{\pi}{480}$, and so it suffices to choose $\alpha_m$ that satisfy
\begin{align*}
    \alpha_m^2 < 8 m^2 e^{- \frac{m}{120}} \lp 1 + 112 e^{\frac{m}{240}} \cos\lp \frac{m}{16} \rp - 15 e^{\frac{m}{240}} \sin\lp \frac{m}{16} \rp \rp.
\end{align*}
For each of the values $1 \leq m \leq 3$, the values $\alpha_1 = 29$, $\alpha_2 = 55$, and $\alpha_3 = 77$ solve the required inequality.
\end{proof}

\begin{lemma} \label{Beta Estimates}
Let $q = e^{-z}$, $z = x + iy$ with $\frac{3\pi}{4} \leq |y| \leq \pi - 15x$ and $0 < x < \frac{\pi}{480}$. Then we have
\begin{align*}
    - \dfrac{e^{-2x}}{2\lp 1 - e^{-4x} \rp} + \dfrac{\cos\lp 2y \rp \lp e^{-2x} - e^{-6x} \rp}{2\left| 1 - q^4 \right|^2} < - \dfrac{1}{10x}. 
\end{align*}
\end{lemma}

\begin{proof}
We have $3\pi \leq |4y| \leq 4\pi - 60x$, and since $\cos\lp y \rp$ is increasing in the region $3\pi \leq y \leq 4y$ we have $\cos\lp 4y \rp \leq \cos\lp 4\pi - 60x \rp = \cos\lp 60 x \rp$. Therefore, we have
\begin{align*}
    \left| 1 - q^4 \right|^2 = 1 - 2 \cos\lp 4y \rp e^{-4x} + e^{-8x} \geq 1 - 2 \cos\lp 60x \rp e^{-4x} + e^{-8x}
\end{align*}
and thus
\begin{align*}
    - \dfrac{e^{-2x}}{2\lp 1 - e^{-4x} \rp} + \dfrac{\cos\lp 2y \rp \lp e^{-2x} - e^{-6x} \rp}{2\left| 1 - q^4 \right|^2} \leq - \dfrac{e^{-2x}}{2\lp 1 - e^{-4x} \rp} + \dfrac{e^{-2x} - e^{-6x}}{2\lp 1 - 2 \cos\lp 60x \rp e^{-4x} + e^{-8x} \rp} =: F(x).
\end{align*}
Fix any $A > 0$. The inequality $F(x) < - \frac{A}{x}$ is equivalent to
\begin{align*}
    2 x e^{-6x}\lp 1 - \cos\lp 60x \rp \rp > 2A\lp 1 - e^{-4x} \rp \lp 1 - 2 \cos\lp 60x \rp e^{-4x} + e^{-8x} \rp.
\end{align*}
If we set $A = \frac{1}{10}$ and rearrange, this is equivalent to showing that
\begin{align*}
    2x e^{-6x} + \dfrac{2}{5} e^{-4x} \cos\lp 60x \rp + \dfrac{1}{5} e^{-4x} + e^{-12x} > \dfrac{1}{5} + 2x e^{-6x} \cos\lp 60x \rp + \dfrac{2}{5} e^{-8x} \cos\lp 60x \rp + \dfrac{1}{5} e^{-8x}.
\end{align*}
By a term-by-term comparison, it would suffice to show that $e^{-12x} > \frac{1}{5}$ for $0 < x < \frac{\pi}{480}$, which is true.
\end{proof}

\begin{lemma} \label{PI^2/12 Lemma}
Let $q = e^{-z}$, $z = x + iy$ with $x>0$ and $15x \leq |y| \leq \pi - 15x$. Then
\begin{align*}
    \Log\lp \lp |q|; |q|^2 \rp_\infty^{-1} \rp < \dfrac{\pi^2}{12x}.
\end{align*}
\end{lemma}

\begin{proof}
We have by expanding series that
\begin{align*}
	\Log\lp \lp |q|; |q|^2 \rp_\infty^{-1} \rp = \sum_{m \geq 1} \dfrac{e^{-mx}}{m\lp 1 - e^{-2mx} \rp}.
\end{align*}
We have $\frac{e^{-x}}{1 - e^{-2x}} < \frac{1}{2x}$; this inequality is equivalent to showing that $2x < e^x - e^{-x}$, which can be proven for all $x > 0$ using elementary calculus. We therefore have by substitutions that
\begin{align*}
	\dfrac{e^{-mx}}{m\lp 1 - e^{-2mx} \rp} < \dfrac{1}{2m^2 x},
\end{align*}
and the result follows by summing over $m$.
\end{proof}

\begin{lemma} \label{Elementary Bounds}
For $1 \leq m \leq 3$ and $0 < x < \frac{\pi}{480}$, we have
\begin{align*}
	\dfrac{e^{-mx}}{m\lp 1 - e^{-2mx}\rp} > \dfrac{499}{1000 m^2 x}.
\end{align*}
\end{lemma}

\begin{proof}
For any $A>0$, the inequality $\frac{e^{-mx}}{1 - e^{-2mx}} > \frac{A}{mx}$ reduces to $mx > A\lp e^{mx} - e^{-mx} \rp$. The left and right-hand sides have equal values at $x=0$, and so by taking derivatives it would suffice to show that $A\lp e^{mx} + e^{-mx} \rp < 1$ for $0 < x < \frac{\pi}{480}$. The left-hand side is now an increasing function of $x$, and so it suffices to check that the inequality is true for $A = \frac{499}{1000}$, $1 \leq m \leq 3$ and $x = \frac{\pi}{480}$, which holds.
\end{proof}

We may now prove the main minor arc bound on $G(q)$.

\begin{prop} \label{Minor Arc Bounds}
Let $q = e^{-z}$ for $z = x + iy$ satisfying $0 < x < \frac{\pi}{480}$ and $15x \leq |y| \leq \pi - 15x$. Then we have
\begin{align*}
    \left| G\lp q \rp \right| < \exp\lp \dfrac{1}{5 x} \rp.
\end{align*}
\end{prop}

\begin{proof}
By taking exponentials, it suffices to prove that $\mathrm{Re}\lp \Log \lp G(q) \rp \rp < \frac{1}{5x}$. As in the proof of Proposition \ref{Major Arc Bound}, we may use Taylor expansions to show
\begin{align*}
    \Log\lp G(q) \rp = \Log\lp\lp q; q^4 \rp_\infty^{-1}\rp + \Log\lp\lp -q^3; q^4 \rp_\infty^{-1}\rp = \sum_{m \geq 1} \dfrac{q^m}{m\lp 1 + \lp -1 \rp^{m+1} q^{2m} \rp}.
\end{align*}
By taking real parts, we have
\begin{align} \label{Exact Real Part}
    \mathrm{Re}\lp \Log\lp G(q) \rp \rp = \sum_{m \geq 1} \dfrac{\cos\lp my \rp \lp |q|^m + \lp -1 \rp^{m+1} |q|^{3m} \rp}{m \left| 1 + \lp -1 \rp^{m+1} q^{2m} \right|^2}.
\end{align}
Note that since cosine is even, we may assume without loss of generality that $y > 0$. This proof uses the idea of ``splitting off terms" in this series expansion. In particular, we make use of the string of inequalities
\begin{align} \label{Real-Aboslute Value Inequality}
    \mathrm{Re}\lp \dfrac{q^m}{m\lp 1 + \lp -1 \rp^{m+1} q^{2m} \rp} \rp \leq \dfrac{|q|^m}{m \left| 1 + \lp -1 \rp^{m+1} q^{2m} \right|} \leq \dfrac{|q|^m}{m \lp 1 - |q|^{2m} \rp},
\end{align}
in order to bound \eqref{Exact Real Part}. A priori, one may show immediately using Lemma \ref{PI^2/12 Lemma} and \eqref{Real-Aboslute Value Inequality} that $\mathrm{Re}\lp \Log\lp G(q) \rp \rp < \frac{\pi^2}{12x}$, which is insufficient for our purposes.  The idea of splitting terms off is to use \eqref{Real-Aboslute Value Inequality} more carefully to keep track of some of the error introduced in this process, eventually pushing the a priori bound of $\frac{\pi^2}{12x}$ below the required $\frac{1}{5x}$. More specifically, by applying \eqref{Real-Aboslute Value Inequality} we have for any integer $k \geq 0$ (with $k=0$ denoting an empty sum) a corresponding ``splitting bound"
\begin{align*}
    \mathrm{Re}\lp \Log\lp G(q) \rp \rp \leq \sum_{m \geq 1} \dfrac{|q|^m}{m\lp 1 - |q|^{2m} \rp} + \sum_{m=1}^k \lp \dfrac{\cos\lp my \rp \lp |q|^m + \lp -1 \rp^{m+1} |q|^{3m} \rp}{m \left| 1 + \lp -1 \rp^{m+1} q^{2m} \right|^2} - \dfrac{|q|^m}{m\lp 1 - |q|^{2m} \rp} \rp.
\end{align*}
The infinite sum is a sort of main term which we must reduce below $\frac{1}{5x}$ by means of the finite sum. By Lemma \ref{PI^2/12 Lemma} we have
\begin{align*}
    \Log\lp \lp |q|; |q|^2 \rp_\infty^{-1} \rp = \sum_{m \geq 1} \frac{|q|^m}{m\lp 1 - |q|^{2m} \rp} < \dfrac{\pi^2}{12x},
\end{align*}
and therefore
\begin{align} \label{Main Splitting Bound}
    \mathrm{Re}\lp \Log\lp G(q) \rp \rp < \dfrac{\pi^2}{12x} &+ \sum_{m=1}^k \lp \dfrac{\cos\lp my \rp \lp e^{-mx} + \lp -1 \rp^{m+1} e^{-3mx} \rp}{m \left| 1 + \lp -1 \rp^{m+1} q^{2m} \right|^2} - \dfrac{e^{-mx}}{m\lp 1 - e^{-2mx} \rp} \rp.
\end{align}
Note that if the conditions of Lemma \ref{Alpha Estimations} are satisfied, then comparison between the first and second terms in \eqref{Real-Aboslute Value Inequality} implies that for any $k \geq \ell \geq 0$ we have
\begin{align} \label{Secondary Splitting Bound}
	\mathrm{Re}\lp \Log\lp G(q) \rp \rp < \dfrac{\pi^2}{12x} &+ \sum_{m=1}^\ell \dfrac{e^{- \frac{m\pi}{480}}}{2m^2 x} \lp \dfrac{2m}{\alpha_m} - 1 \rp \notag \\ &+ \sum_{m=\ell+1}^k \lp \dfrac{\cos\lp my \rp \lp e^{-mx} + \lp -1 \rp^{m+1} e^{-3mx} \rp}{m \left| 1 + \lp -1 \rp^{m+1} q^{2m} \right|^2} - \dfrac{e^{-mx}}{m\lp 1 - e^{-2mx} \rp} \rp.
\end{align}

Our objective now is to prove that the right-hand side of either \eqref{Main Splitting Bound} or \eqref{Secondary Splitting Bound} is bounded above by $\frac{1}{5x}$ for all $0 < x < \frac{\pi}{480}$ and all $15x \leq y \leq \pi - 15x$. This will not be done all at once, but in stages. In the first stage of the proof, we will split the interval $\frac{\pi}{2} \leq y \leq \pi - 15x$ into several subintervals. On each subinterval, some version of \eqref{Main Splitting Bound} will be sufficient to prove the desired inequality. After this is completed, we will be able to apply the $\ell = 1$ case of \eqref{Secondary Splitting Bound}. We will use this case to prove the result in the range $\frac{\pi}{4} \leq y < \frac{\pi}{2}$. We then use the case $\ell = 2$ of \eqref{Secondary Splitting Bound} to cover the range $\frac{\pi}{6} \leq y < \frac{\pi}{4}$, and finally we will use the case $\ell = 3$ of \eqref{Secondary Splitting Bound} to cover the range $15x \leq y < \frac{\pi}{6}$. All of these cases together prove the desired result in the full range $15x \leq y \leq \pi - 15x$. We begin now with the application of \eqref{Main Splitting Bound} to the interval $\frac{\pi}{2} \leq y \leq \pi - 15x$. \\

Suppose $\frac{5\pi}{6} \leq y \leq \pi - 15x$. Because $\cos\lp y \rp, \cos\lp 3y \rp \leq 0$ in this range, we have using the $k = 3$ case of \eqref{Main Splitting Bound} that
\begin{align*}
    \mathrm{Re}\lp \Log\lp G(q) \rp \rp < \dfrac{\pi^2}{12x} - \dfrac{e^{-x}}{1 - e^{-2x}} + \dfrac{\cos\lp 2y \rp \lp e^{-2x} - e^{-6x} \rp}{2\lp 1 - \cos\lp 4y \rp e^{-4x} + e^{-8x} \rp} - \dfrac{e^{-2x}}{2\lp 1 - e^{-4x} \rp} - \dfrac{e^{-3x}}{3\lp 1 - e^{-6x} \rp}.
\end{align*}
By Lemmas \ref{Beta Estimates} and \ref{Elementary Bounds}, we therefore have
\begin{align*}
    \mathrm{Re}\lp \Log\lp G(q) \rp \rp < \dfrac{\pi^2}{12x} - \dfrac{1}{10x} - \dfrac{e^{-x}}{1 - e^{-2x}} - \dfrac{e^{-3x}}{3\lp 1 - e^{-6x} \rp} < \lp \dfrac{\pi^2}{12} - \dfrac{1}{10} - \dfrac{499}{1000} \lp 1 + \dfrac{1}{9} \rp \rp \dfrac{1}{x},
\end{align*}
for all $0 < x < \frac{\pi}{480}$, which establishes $\mathrm{Re}\lp \Log\lp G(q) \rp \rp < \frac{1}{5x}$ in this region.

We now consider the region $\frac{3\pi}{4} \leq y < \frac{5\pi}{6}$. In this region we have $\cos\lp y \rp \leq 0$, and so by the $k = 2$ variant of \eqref{Main Splitting Bound} we have
\begin{align*}
    \mathrm{Re}\lp \Log\lp G(q) \rp \rp < \dfrac{\pi^2}{12x} - \dfrac{e^{-x}}{1 - e^{-2x}} + \dfrac{\cos\lp 2y \rp \lp e^{-2x} - e^{-6x} \rp}{2\lp 1 - 2 \cos\lp 4y \rp e^{-4x} + e^{-8x} \rp} - \dfrac{e^{-2x}}{2\lp 1 - e^{-4x} \rp}.
\end{align*}
By considering partial derivatives of the numerator and denominator separately, we can see that in the region $\frac{3\pi}{4} \leq y < \frac{5\pi}{6}$ the fraction $\frac{\cos\lp 2y \rp \lp e^{-2x} - e^{-6x} \rp}{2\lp 1 - \cos\lp 4y \rp e^{-4x} + e^{-8x} \rp}$ is an increasing function of $y$, and therefore we have in this region by applying Lemma \ref{Elementary Bounds} that
\begin{align*}
    \mathrm{Re}\lp \Log\lp G(q) \rp \rp &< \dfrac{\pi^2}{12x} - \dfrac{e^{-x}}{1 - e^{-2x}} - \dfrac{e^{-2x}}{2\lp 1 - e^{-4x} \rp} + \dfrac{e^{-2x} - e^{-6x}}{4\lp 1 - e^{-4x} + e^{-8x} \rp} \\ &< \lp \dfrac{\pi^2}{12} - \dfrac{499}{1000} \lp 1 + \dfrac{1}{4} \rp \rp \dfrac{1}{x} + \dfrac{e^{-2x} - e^{-6x}}{4\lp 1 - e^{-4x} + e^{-8x} \rp}.
\end{align*}
It is clear that the term $\frac{e^{-2x} - e^{-6x}}{4\lp 1 - e^{-4x} + e^{-8x} \rp}$ is extremely small in $0 < x < \frac{\pi}{480}$. In particular, it is straightforward to show that this quantity is less than $\frac{7}{1000}$ for $0 < x < \frac{\pi}{480}$. It follows that
\begin{align*}
	\lp \dfrac{\pi^2}{12} - \dfrac{499}{1000} \lp 1 + \dfrac{1}{4} \rp \rp \dfrac{1}{x} + \dfrac{e^{-2x} - e^{-6x}}{4\lp 1 - e^{-4x} + e^{-8x} \rp} < \lp \dfrac{\pi^2}{12} - \dfrac{499}{1000} \lp 1 + \dfrac{1}{4} \rp \rp \dfrac{1}{x} + \dfrac{7}{1000} <  \dfrac{1}{5x}
\end{align*}
for $0 < x < \frac{\pi}{480}$ and therefore $\mathrm{Re}\lp \Log\lp G(q) \rp \rp < \frac{1}{5x}$ for $0 < x < \frac{\pi}{480}$ and $\frac{3\pi}{4} \leq y < \frac{5\pi}{6}$.

Consider now the range $\frac{\pi}{2} \leq y < \frac{3\pi}{4}$. Here, we have $\cos\lp y \rp, \cos\lp 2y \rp \leq 0$ and therefore by the $k = 2$ case of \eqref{Main Splitting Bound} and Lemma \ref{Elementary Bounds} we obtain
\begin{align*}
    \mathrm{Re}\lp \Log\lp G(q) \rp \rp < \dfrac{\pi^2}{12x} - \dfrac{e^{-x}}{1 - e^{-2x}} - \dfrac{e^{-2x}}{2\lp 1 - e^{-4x} \rp} < \lp \dfrac{\pi^2}{12} - \dfrac{49}{100}\lp 1 + \dfrac{1}{4} \rp \rp \dfrac{1}{x}
\end{align*}
As in the previous case, this establishes $\mathrm{Re}\lp \Log\lp G(q) \rp \rp < \frac{1}{5x}$ for all $0 < x < \frac{\pi}{480}$ and, by taking together all previous cases as well as this one, all $\frac{\pi}{2} \leq y < \pi$.

Note that we are reduced to the region $15x \leq y < \frac{\pi}{2}$, and so we may invoke the case $\ell = 1$ of \eqref{Secondary Splitting Bound}. Consider the range $\frac{\pi}{4} \leq y < \frac{\pi}{2}$. By the $\ell=1$, $k=3$ case of \eqref{Secondary Splitting Bound} along with Lemma \ref{Elementary Bounds} and the fact that $\cos\lp 2y \rp, \cos\lp 3y \rp \leq 0$ in this region, we obtain
\begin{align*}
    \mathrm{Re}\lp \Log\lp G(q) \rp \rp &< \dfrac{\pi^2}{12x} - \dfrac{27 e^{-\frac{\pi}{480}}}{58x} - \dfrac{e^{-2x}}{2\lp 1 - e^{-4x} \rp} - \dfrac{e^{-3x}}{3\lp 1 - e^{-6x} \rp} \\ &< \lp \dfrac{\pi^2}{12} - \dfrac{27 e^{-\frac{\pi}{480}}}{58} - \dfrac{499}{1000}\lp \dfrac{1}{4} + \dfrac 19 \rp \rp \dfrac{1}{x},
\end{align*}
which is less than $\frac{1}{5x}$, so the desired result is proven in the region $\frac{\pi}{4} \leq y < \frac{\pi}{2}$.

We now consider the range $\frac{\pi}{6} \leq y < \frac{\pi}{4}$, within which the $\ell = 2$ case of \eqref{Secondary Splitting Bound} applies by Lemma \ref{Alpha Estimations}. By \eqref{Secondary Splitting Bound} with $\ell = 2$ and $k = 3$, we have
\begin{align*}
    \mathrm{Re}\lp \Log\lp G(q) \rp \rp < \dfrac{\pi^2}{12x} &- \dfrac{27 e^{-\frac{\pi}{480}}}{58x} - \dfrac{51 e^{- \frac{\pi}{240}}}{440x} + \dfrac{\cos\lp 3y \rp \lp e^{-3x} + e^{-9x} \rp}{3 \left| 1 + q^6 \right|^2} - \dfrac{e^{-3x}}{3\lp 1 - e^{-6x} \rp}.
\end{align*}
Since in this range we have $\cos\lp 3y \rp \leq 0$, we have
\begin{align*}
    \mathrm{Re}\lp \Log\lp G(q) \rp \rp < \dfrac{\pi^2}{12x} &- \dfrac{27 e^{-\frac{\pi}{480}}}{58x} - \dfrac{51 e^{- \frac{\pi}{240}}}{440x} - \dfrac{e^{-3x}}{3\lp 1 - e^{-6x} \rp},
\end{align*}
which is as in earlier cases yields the desired result for $0 < x < \frac{\pi}{480}$ by Lemma \ref{Elementary Bounds}.

Finally, consider the interval $0 < 15x \leq y < \frac{\pi}{6}$. We may use case $\ell = k = 3$ of \eqref{Secondary Splitting Bound}, which implies
\begin{align*}
    \mathrm{Re}\lp \Log\lp G(q) \rp \rp < \dfrac{\pi^2}{12x} - \dfrac{21 e^{-\frac{\pi}{480}}}{58x} - \dfrac{51 e^{-\frac{\pi}{240}}}{440x} - \dfrac{71 e^{-\frac{\pi}{160}}}{1386x}.
\end{align*}
for $0 < x < \frac{\pi}{480}$. The right-hand side above is always less than $\frac{1}{5x}$, and this completes the proof in the region $15x \leq y < \frac{\pi}{6}$. This completes the proof of the proposition.
\end{proof}

\subsection{Bounds on $F_a^{r,t}(z)$ and $E^{r,t}(z)$}

We will need the following effective estimates of the functions $F_a^{r,t}(z)$ and $E^{r,t}(z)$ which appear in Lemma \ref{B_rt Bound}.

\begin{lemma} \label{E-Bounds}
Let $0 < a \leq 1$ be a real number and $z = x + iy$ any complex number satisfying $|z| < 1$ and $0 \leq |y| < 15x$. Then we have
\begin{align*}
    E^{1,4}(z) < 28 |z|^3
\end{align*}
and
\begin{align*}
	E^{3,4}(z) < 56 |z|^3.
\end{align*}
\end{lemma}

\begin{proof}
Recall that
\begin{align*}
	E^{r,t}(z) =  \dfrac{J_{g_{r,t},4}(z)}{720} |z|^3 + \sum_{k \geq 3} \left| \dfrac{B_{k+2}\lp 1 - \frac rt \rp}{(k+2)!} - \dfrac{\lp -r \rp^{k+1} \lp \frac 12 - \frac rt \rp}{t^{k+1} (k+1)!} \right| \lp 1 + \dfrac{k}{10(k-3)!} \rp |z|^k,
\end{align*}
where
\begin{align*}
	c_n^* = \begin{cases} \dfrac{B_{n+1}\lp 1 - \frac rt \rp}{(n+2)!} & \text{ if } n \leq 2, \\ \dfrac{(-r)^{n+1} \lp \frac 12 - \frac rt \rp}{t^{n+1} (n+1)!} & \text{ otherwise}. \end{cases}
\end{align*}
We first consider the two integrals $J_{g_{1,4},4}(z)$ and $J_{g_{3,4},4}(z)$, which we recall are taken over a path of integration going through the origin and $z$. We bound these integrals by splitting them each into upper and lower parts, taking advantage of the decay properties of $g_{r,t}(x)$ in the upper parts and power series expansions in the lower parts. For both cases $r = 1,3$, we have
\begin{align*}
	g^{(4)}_{r,t}(w) = \dfrac{e^{-\frac{rw}{4}}}{1024\lp e^w - 1 \rp^5 w^6} \lp \sum_{j=0}^5 e^{jw} p_{r,t,j}(w) + \sum_{j=0}^5 \tilde{c}_j e^{\frac{4j+r}{4} w} \rp
\end{align*}
for certain constants $\tilde{c}_j$ and degree 5 polynomials $p_{r,t,j}(w)$. For $\alpha = \frac{3\pi}{2} \frac{z}{|z|}$, applying the triangle inequality and the major arc bounds $\mathrm{Re}(w) \leq |w| \leq \sqrt{226}\mathrm{Re}(w)$ imply upper bounds on $\left| g^{(4)}_{r,4}(w) \right|$ that depend only on $u = \mathrm{Re}(w)$. Using these upper bounds, we can conclude that
\begin{align*}
	\int_\alpha^\infty \left| g^{(4)}_{1,4}(w) \right| |dw| < 19900, \ \ \ \text{ and } \ \ \ \int_\alpha^\infty \left| g^{(4)}_{3,4}(w) \right| |dw| < 39900.
\end{align*}
To bound the remainder of the integrals $J_{g_{r,4},4}(z)$, we use the power series representations of $g_{r,4}^{(4)}(w)$, namely
\begin{align*}
	g_{1,4}^{(4)}(w) = \sum_{n \geq 0} \dfrac{(n+4)!}{n!} \lp \dfrac{B_{n+6}\lp \frac 34 \rp}{(n+6)!} + \dfrac{(-1)^n}{4^{n+6} (n+5)!} \rp w^n
\end{align*}
and
\begin{align*}
	g_{3,4}^{(4)}(w) = \sum_{n \geq 0} \dfrac{(n+4)!}{n!} \lp \dfrac{B_{n+6}\lp \frac 14 \rp}{(n+6)!} + \dfrac{(-1)^{n+1} 3^{n+5}}{4^{n+6}} \rp w^n.
\end{align*}
By applying \eqref{Lehmer's Bound}, $\zeta(n+6) \leq \frac{\pi^6}{945}$, $|w| < \frac{3\pi}{2}$ and other elementary estimates, we have
\begin{align*}
	\left| g^{(4)}_{1,4}(w) \right| \leq \sum_{n \geq 0} \dfrac{(n+4)!}{n!} \lp \dfrac{2\pi^6}{945 (2\pi)^{n+6}} + \dfrac{1}{4^{n+6} (n+5)!} \rp \lp \dfrac{3\pi}{2} \rp^n < 1
\end{align*}
and
\begin{align*}
	\left| g^{(4)}_{3,4}(w) \right| \leq \sum_{n \geq 0} \dfrac{(n+4)!}{n!} \lp \dfrac{2\pi^6}{945 (2\pi)^{n+6}} + \dfrac{3^{n+5}}{4^{n+6} (n+5)!} \rp \lp \dfrac{3\pi}{2} \rp^n < 2.
\end{align*}
Therefore, for $\alpha = \frac{3\pi}{2} \frac{z}{|z|}$, we have
\begin{align*}
	J_{g_{1,4},4}(z) = \int_0^\alpha \left| g^{(4)}_{1,4}(w) \right| |dw| + \int_\alpha^\infty \left| g^{(4)}_{1,4}(w) \right| |dw| < 20000
\end{align*}
and
\begin{align*}
	J_{g_{3,4},4}(z) = \int_0^\alpha \left| g^{(4)}_{3,4}(w) \right| |dw| + \int_\alpha^\infty \left| g^{(4)}_{3,4}(w) \right| |dw| < 40000.
\end{align*}
We now bound the other summand of $E^{r,4}(z)$ in the cases $r = 1,3$. Using \eqref{Lehmer's Bound}, along with $\zeta(n) \leq \frac{\pi^2}{6}$ for $n \geq 2$ and $|z| < 1$, we have
\begin{align*}
    \sum_{k \geq 3} \left| \dfrac{B_{k+2}\lp 1 - \frac 14 \rp}{(k+2)!} - \dfrac{(-1)^{n+1}}{4^{k+2} (k+1)!} \right| \lp 1 + \dfrac{k!}{10(k-3)!} \rp |z|^k < \dfrac{|z|^3}{1000}
\end{align*}
and
\begin{align*}
	\sum_{k \geq 3} \left| \dfrac{B_{k+2}\lp 1 - \frac 34 \rp}{(k+2)!} - \dfrac{(-3)^{k+1}}{4^{k+2} (k+1)!} \right| \lp 1 + \dfrac{k!}{10(k-3)!} \rp |z|^k < \dfrac{|z|^3}{100}.
\end{align*}
Therefore,
\begin{align*}
    \left| E^{1,4}(z) \right| < \dfrac{20000}{720} |z|^3 + \dfrac{1}{1000} |z|^3 < 28 |z|^3
\end{align*}
and
\begin{align*}
    \left| E^{3,4}(z) \right| < \dfrac{40000}{720} |z|^3 + \dfrac{1}{100} |z|^3 < 56 |z|^3,
\end{align*}
which completes the proof.
\end{proof}

We now estimate a certain combination of the functions $F_a^{r,t}(z)$ in a similar manner. Define the functions $G_1^*(q), G_2^*(q)$ respectively by
\begin{align*}
    G_1^*(q) := \exp\lp \dfrac{\pi^2}{48z} - \dfrac{1}{4} \Log\lp z \rp + \beta_{1,4} - \dfrac{\log(2)}{4} - \dfrac{z}{24} \rp
\end{align*}
and
\begin{align*}
    G_2^*(-q) := \exp\lp \dfrac{\pi^2}{48z} + \dfrac{1}{4} \Log\lp z \rp + \beta_{3,4} + \dfrac{\log(2)}{4} - \dfrac{z}{24} \rp.
\end{align*}
These will be useful in estimating $G(q)$ along the two major arcs in the circle method.

\begin{lemma} \label{F-Bounds}
Let $q = e^{-z}$, $z = x + iy$ satisfy $0 < x < \frac{\pi}{480}$ and $0 \leq |y| < 15x$.
\begin{enumerate}
    \item We have
    \begin{align*}
        \left| 4z F_1^{1,4}(4z) + 4z F_1^{3,4}(8z) - 4z F_{1/2}^{3,4}(8z) - \Log\lp G_1^*(q) \rp \right| \leq \dfrac{|z|^4}{2}.
    \end{align*}
    \item We have
    \begin{align*}
        \left| 4z F_1^{3,4}(4z) + 4z F_1^{1,4}(8z) - 4z F_{1/2}^{1,4}(8z) - \Log\lp G_2^*(-q) \rp \right| \leq \dfrac{|z|^4}{2}.
    \end{align*}
    \end{enumerate}
\end{lemma}

\begin{proof}
Define the functions
\begin{align*}
    F_1(z) := 4z \lp F_1^{1,4}(4z) + F_1^{3,4}(8z) - F_{1/2}^{3,4}(8z) \rp
\end{align*}
and
\begin{align*}
    F_2(z) := 4z\lp F_1^{3,4}(4z) + F_1^{1,4}(8z) - F_{1/2}^{1,4}(8z) \rp.
\end{align*}
By expanding each of the terms $F_a^{r,t}(z)$, we have series expansions
\begin{align*}
    F_1(z) - \Log\lp G_1^*(q) \rp = \sum_{n \geq 3} \alpha_{n+1} z^{n+1}
\end{align*}
and
\begin{align*}
    F_2(z) - \Log\lp G_2^*(-q) \rp = \sum_{n \geq 3} \alpha^\prime_{n+1} z^{n+1}
\end{align*}
where
\begin{align*}
    \alpha_{n+1} = (-1)^{n+1} \dfrac{B_{n+1}(1) + 3 \cdot 6^n \lp B_{n+1}\lp \frac 12 \rp - B_{n+1}(1) \rp}{4 (n+1) \cdot (n+1)!}
\end{align*}
and
\begin{align*}
    \alpha_{n+1}^\prime = \lp -1 \rp^{n+1} \dfrac{- 3^{n+1} B_{n+1}(1) + 2^n\lp B_{n+1}(1) - B_{n+1}\lp \frac 12 \rp \rp}{4 (n+1) (n+1)!}.
\end{align*}
Now, by \eqref{Lehmer's Bound}, we have for $n \geq 2$ that $M_n = \max\limits_{0 \leq x \leq 1} \left| B_n(x) \right| \leq \frac{2 \zeta(n) n!}{\lp 2\pi \rp^n}$,  we have for $n \geq 1$ the bounds
\begin{align*}
    \left| \alpha_{n+1} \right| \leq \dfrac{\left| B_{n+1}(1) \right| + 3 \cdot 6^n \lp \left| B_{n+1}(1) \right| + \left| B_{n+1}\lp \frac 12 \rp \right| \rp}{4 (n+1) \cdot (n+1)!} &\leq \dfrac{(1 + 6^{n+1}) M_{n+1}}{4 (n+1) \cdot (n+1)!} \\ &< \dfrac{\pi^2}{12(n+1)} \lp \dfrac{1}{\lp 2\pi \rp^{n+1}} + \lp \dfrac{3}{\pi} \rp^{n+1} \rp
\end{align*}
and likewise
\begin{align*}
    \left| \alpha^\prime_{n+1} \right| \leq \dfrac{\lp 3^{n+1} + 2^{n+1} \rp M_{n+1}}{4 (n+1) \cdot (n+1)!} < \dfrac{\pi^2}{12(n+1)} \lp \dfrac{3}{\pi} \rp^{n+1}.
\end{align*}
Therefore, noting that on the major arc $0 \leq |y| < 15x$ with $0 < x < \frac{\pi}{480}$ we have $|z| < \frac{\sqrt{226} \pi}{480}$, we have
\begin{align*}
    \left| F_1(z) - \Log\lp G_1^*(q) \rp \right| \leq \sum_{n \geq 3} \left| \alpha_{n+1} \right| |z|^{n+1} < \dfrac{|z|^4}{2} 
\end{align*}
and likewise
\begin{align*}
    \left| F_2(z) - \Log\lp G_2^*(-q) \rp \right| \leq \sum_{n \geq 3} \left| \alpha^\prime_{n+1} \right| |z|^{n+1} < \dfrac{|z|^4}{2}.
\end{align*}
This completes the proof.
\end{proof}

We now use the bounds so far derived to give an estimate for $G(q)$ on arcs near $q = \pm 1$. For this final lemma, we require some new notation. For any complex-valued function $f(z)$ and any real-valued function $g(z)$, we shall say that $f(z) = O_{\leq}\lp g(z) \rp$ if $\left| f(z) \right| \leq g(z)$ for all $z$ in a specified region (which will always be clear from context).

\begin{lemma} \label{Major Error Bound}
Let $q = e^{-z}$, $z = x + iy$, $0 < x < \frac{\pi}{480}$ and $0 \leq |y| < 15x$.
\begin{enumerate}
    \item We have $\Log\lp G(q) \rp = \Log\lp G_1^*(q) \rp + E_+(q)$ where $E_+(q) = O_{\leq}\lp 4033 |z|^4 \rp$.
    \item We have $\Log\lp G(-q) \rp = \Log\lp G_2^*(-q) \rp + E_-(q)$ where $E_-(q) = O_{\leq}\lp 2689 |z|^4 \rp$.
\end{enumerate}
\end{lemma}

\begin{proof}
Let $F_1(z), F_2(z)$ be defined as in the proof of Lemma \ref{F-Bounds}. Then by Proposition \ref{Major Arc Bound} and Lemma \ref{E-Bounds},  we have
\begin{align*}
    \left| \Log\lp G(q) \rp - F_1(z) \right| < 4|z| E^{1,4}(4z) + 8|z| E^{3,4}(8z) < 4032 |z|^4
\end{align*}
and similarly
\begin{align*}
    \left| \Log\lp G(-q) \rp - F_2(z) \right| < 2688 |z|^4.
\end{align*}
in the relevant domain. We therefore obtain by Lemma \ref{F-Bounds} (1) and (2) that
\begin{align*}
    \left| \Log\lp G(q) \rp - \Log\lp G_1^*(q) \rp \right| < 4033 |z|^4
\end{align*}
and
\begin{align*}
    \left| \Log\lp G(-q) \rp - \Log\lp G_2^*(-q) \rp \right| < 2689 |z|^4.
\end{align*}
This proves the result.
\end{proof}

\section{Proof of Theorems \ref{MAIN} and \ref{a(n) Asymptotics}} \label{Proof of MAIN}

We now proceed to the proof of Theorem \ref{MAIN} (Theorem \ref{a(n) Asymptotics} will be proven along the way), which relies on a variation of Wright's circle method. We set $q = e^{-z}$ with $\mathrm{Re}\lp z \rp = x$ and $\mathrm{Im}\lp z \rp = y$. Since $G(q)$ has no poles inside the unit disk, we have by Cauchy's theorem that
\begin{align*}
    a(n) = \dfrac{1}{2\pi i} \int_C \dfrac{G(q)}{q^{n+1}} dq,
\end{align*}
where $C$ is the circle oriented counterclockwise centered at 0 with radius $|q| = e^{-x}$. We choose $C$ so that $x = \frac{\pi}{\sqrt{48n}}$. Impose the constraint $0 < x < \frac{\pi}{480}$ throughout, which by algebraic manipulations is equivalent to the assumption that $n > 4800$. Define the three arcs 
\begin{center}
\begin{align*}
	C_1 &:= \{ q = e^{-z} \in C : 0 \leq |y| < 15 x \}, \\ C_2 &:= \{ q = e^{-z} \in C : \pi - 15x \leq |y| < \pi \},
\end{align*}
\end{center}
and
\begin{align*}
	\widetilde{C} := \{ q = e^{-z} \in C : 15x \leq |y| \leq \pi - 15x \}.
\end{align*}
We may decompose $a(n)$ as
\begin{align*}
    a(n) =  J^*_1(n) + J^*_2(n) + J^{\text{maj}}_1(n) + J^{\text{maj}}_2(n) + J^{\text{min}}(n),
\end{align*}
where for $k = 1, 2$ we define
define
\begin{align*}
	\widetilde{G}_1(q) := \exp\lp \dfrac{\pi^2}{48z} - \dfrac{1}{4} \Log\lp z \rp + \beta_{1,4} - \dfrac{\log(2)}{4} \rp, \hspace{0.1in} \widetilde{G}_2(-q) := \exp\lp \dfrac{\pi^2}{48z} + \dfrac{1}{4} \Log\lp z \rp + \beta_{3,4} + \dfrac{\log(2)}{4} \rp,
\end{align*}
and
\begin{align*}
    J^*_k(n) := \dfrac{1}{2\pi i} \int_{C_k} \dfrac{\widetilde{G}_k(q)}{q^{n+1}} dq, \hspace{0.2in} J^{\text{maj}}_k(n) := \dfrac{1}{2\pi i} \int_{C_k} \dfrac{G(q) - \widetilde{G}_k(q)}{q^{n+1}} dq, \hspace{0.2in} J^{\text{min}}(n) := \dfrac{1}{2\pi i} \int_{\widetilde C} \dfrac{G(q)}{q^{n+1}} dq.
\end{align*}
$J_1^*(n)$ and $J_2^*(n)$ are the dominant terms, and so we begin with an analysis of the error terms.

\subsection{Error Bound for $J^{\text{min}}(n)$}

By Proposition \ref{Minor Arc Bounds}, we have on all $\tilde C$ that $\left| G(q) \right| < \exp\lp \frac{1}{5x} \rp$. Since the length of $\widetilde C$ is less than $2\pi$ and $\left| \int_{\widetilde{C}} q^{-1} dq \right| < 2\pi |q|^{-1} = 2 \pi e^{\frac{\pi}{480}} < \frac{21\pi}{10}$, it follows that
\begin{align*}
    \left| J^{\text{min}}(n) \right| = \left| \dfrac{1}{2\pi i} \int_{\widetilde C} \dfrac{G(q)}{q^{n+1}} dq \right| < \frac{21\pi}{10} \exp\lp nx + \dfrac{1}{5x} \rp = \frac{21\pi}{10} \exp\lp \lp \dfrac{\pi}{4\sqrt{3}} + \dfrac{4\sqrt{3}}{5\pi} \rp \sqrt{n} \rp.
\end{align*}

\subsection{Error Bounds for $J_1^{\text{maj}}(n)$ and $J_2^{\text{maj}}(n)$}

We now consider $J^{\text{maj}}_1(n)$ and $J^{\text{maj}}_2(n)$. For $J_1^{\text{maj}}(n)$, we may assume now that $0 \leq |y| < 15x$. Since we have $x \leq |z| \leq \sqrt{226} x$,
\begin{align*}
    \left| \widetilde{G}_1(q) \right| = \dfrac{\Gamma\lp \frac 14 \rp}{2^{\frac 34} \pi^{\frac 12}} |z|^{- \frac 14} \exp\lp \dfrac{\pi^2x}{48|z|^2} \rp \leq \dfrac{\Gamma\lp \frac 14 \rp}{2^{\frac 34} \pi^{\frac 12}} x^{- \frac 14} \exp\lp \dfrac{\pi^2}{48x} \rp
\end{align*}
and
\begin{align*}
    \left| \widetilde{G}_2(-q) \right| = \dfrac{\Gamma\lp \frac 34 \rp}{2^{\frac 14} \pi^{\frac 12}} |z|^{\frac 14} \exp\lp \dfrac{\pi^2x}{48|z|^2} \rp \leq \dfrac{226^{\frac 18} \Gamma\lp \frac 34 \rp}{2^{\frac 14} \pi^{\frac 12}} x^{\frac 14} \exp\lp \dfrac{\pi^2}{48x} \rp.
\end{align*}
Similarly, we have
\begin{align*}
	\left|G_1^*(q)\right| = \left|\widetilde{G}_1(q)\right| \cdot \left|\exp\lp - \dfrac{z}{24} \rp \right| = \left|\widetilde{G}_1(q)\right| \cdot \exp\lp - \dfrac{x}{24} \rp < \left| \widetilde{G}_1(q) \right|
\end{align*}
and $\left|G_2^*(-q)\right| < \left|\widetilde{G}_2(-q)\right|$. By Lemma \ref{Major Error Bound}, we have
\begin{align*}
    \Log\lp G(q) \rp - \Log\lp G_1^*(q) \rp = E_+(q) = O_{\leq}\lp 4033 |z|^4 \rp,
\end{align*}
and therefore by exponentiation $G(q) = G_1^*(q) \exp\lp E_+(q) \rp$. Now, since $|z| < \frac{\sqrt{226} \pi}{480}$ on $C_1$, we have $\left| E_+(q) \right| < 4033 |z|^4 < 0.38$. Because $\exp\lp t \rp = 1 + O_{\leq}\lp 2 t \rp$ for $0 < t < 0.76$, we have
\begin{align*}
    \left| \exp\lp E_+(q) \rp - 1 \right| < \dfrac{3}{2} \left| E_+(q) \right| < 6050 |z|^4.
\end{align*}
In particular, this implies
\begin{align*}
    \left| G(q) - G_1^*(q) \right| = \left| G_1^*(q) \right| \cdot \left| \exp\lp E_+(q) \rp - 1 \right| < \dfrac{9833929}{n^{\frac{15}{8}}} \exp\lp \dfrac{\pi}{4} \sqrt{\dfrac{n}{3}} \rp.
\end{align*}
We now make a similar estimate for $G_1^*(q) - \widetilde{G}_1(q)$. It is clear from definitions that 
\begin{align*}
	\Log\lp G_1^*(q) \rp - \Log\lp \widetilde{G}_1(q) \rp = - \frac{z}{24} = O_{\leq}\lp \frac{1}{24} |z| \rp,
\end{align*}
and so reasoning as earlier, we may write $G_1^*(q) = \widetilde{G}_1(q) \exp\lp -\frac{z}{24} \rp$. We have $\exp\lp t \rp = 1 + O_{\leq}\lp \frac{12}{11} t \rp$ for $0 < t < 0.005$, so since $\frac{|z|}{24} < 0.005$ we have $\exp\lp -\frac{z}{24} \rp = 1 + O_{\leq}\lp \frac{1}{22} |z| \rp$. Therefore,
\begin{align*}
	\left|G_1^*(q) - \widetilde{G}_1(q)\right| \leq \dfrac{|z|\Gamma\lp \frac 14 \rp}{22 \cdot 2^{\frac 34} \pi^{\frac 12}} x^{- \frac 14} \exp\lp \dfrac{\pi^2}{48x} \rp < \dfrac{1}{2 n^{\frac 38}} \exp\lp \dfrac{\pi}{4} \sqrt{\dfrac{n}{3}} \rp.
\end{align*}
Thus, on $C_1$ we have
\begin{align*}
	\left|G(q) -\widetilde{G}_1(q)\right| \leq \lp \dfrac{1}{2 n^{\frac 38}} + \dfrac{9833929}{n^{\frac{15}{8}}} \rp \exp\lp \dfrac{\pi}{4} \sqrt{\dfrac{n}{3}} \rp.
\end{align*}
Now, let $D_0 = \{ z\in \C : \mathrm{Re}\lp z \rp = x, \left| \mathrm{Im}\lp z \rp \right| \leq 15 x \}$, which is the image of $C_1$ under the change of variables $q \mapsto z$. Since $D_0$ has length $30x$, we have
\begin{align*}
	\left|J_1^{\text{maj}}(n)\right| \leq \dfrac{1}{2\pi} \int_{C_1} \dfrac{\left| G(q) - \widetilde{G}_1(q) \right|}{|q|^{n+1}} dq &\leq \dfrac{1}{2\pi} \int_{D_0} \left| G(q) -  \widetilde{G}_1(q) \right| \left| \exp\lp nz \rp \right| |dz| \\ &\leq \dfrac{30x}{2\pi} \cdot \lp \dfrac{1}{2 n^{\frac 38}} + \dfrac{9833929}{n^{\frac{15}{8}}} \rp \exp\lp \dfrac{\pi}{4} \sqrt{\dfrac{n}{3}} + nx \rp \\ &< \lp \dfrac{2}{n^{\frac{7}{8}}} + \dfrac{21291081}{n^{\frac{19}{8}}} \rp \exp\lp \dfrac{\pi}{2} \sqrt{\dfrac{n}{3}} \rp.
\end{align*}
We may similarly analyze the case of $G(q) - G_2^*(q)$. Note to begin that we may shift $C_2$ to $C_1$ by the substitution $q \mapsto -q$, and so
\begin{align*}
	\left|J^{\text{maj}}_2(n)\right| \leq \left| \dfrac{1}{2\pi i} \int_{C_1} \dfrac{G(-q) - \widetilde{G}_2(-q)}{q^{n+1}} dq \right|.
\end{align*}
We have by Lemma \ref{Major Error Bound} that $\Log\lp G(-q) \rp - \Log\lp G_2^*(-q) \rp = E_-(q) = O_{\leq}\lp 2689 |z|^4 \rp$. Thus $\left|E_-(q)\right| < \frac{3}{10}$ and as before we have $\exp\lp t \rp = 1 + O_{\leq} \lp \frac{3}{2}t \rp$. Thus, $\exp\lp E_-(q) \rp = 1 + O_{\leq}\lp 4034 |z|^4 \rp$, and by the same reasoning as in the first case we obtain
\begin{align*}
	\left| G(-q) - G_2^*(-q) \right| < \dfrac{8183085}{n^{\frac{17}{8}}} \exp\lp \dfrac{\pi}{4} \sqrt{\dfrac{n}{3}} \rp.
\end{align*}
As in the previous case, we have $G_2^*(-q) - \widetilde{G}_2(-q) = \widetilde{G}_2(-q) \cdot O_{\leq}\lp \frac{1}{22} |z| \rp$, and therefore
\begin{align*}
	\left|G_2^*(-q) - \widetilde{G}_2(-q)\right| \leq \dfrac{|z|}{22} \cdot \dfrac{226^{\frac 18} \Gamma\lp \frac 34 \rp}{2^{\frac 14} \pi^{\frac 12}} x^{\frac 14} \exp\lp \dfrac{\pi^2}{48x} \rp \leq \dfrac{3}{10 n^{\frac 58}} \exp\lp \dfrac{\pi}{4} \sqrt{\dfrac{n}{3}} \rp.
\end{align*}
Combining the two cases,
\begin{align*}
	\left|G(-q) - \widetilde{G}_2(-q)\right| \leq \lp \dfrac{3}{10 n^{\frac 58}} + \dfrac{8183085}{n^{\frac{17}{8}}} \rp \exp\lp \dfrac{\pi}{4} \sqrt{\dfrac{n}{3}} \rp
\end{align*}
and therefore
\begin{align*}
	\left|J^{\text{maj}}_2(n)\right| &\leq \dfrac{30x}{2\pi} \cdot \lp \dfrac{3}{10 n^{\frac 58}} + \dfrac{8183085}{n^{\frac{17}{8}}} \rp \exp\lp \dfrac{\pi}{4} \sqrt{\dfrac{n}{3}} + nx \rp \\ &< \lp \dfrac{13}{20 n^{\frac 98}} + \dfrac{17716899}{n^{\frac{21}{8}}} \rp \exp\lp \dfrac{\pi}{2} \sqrt{\dfrac{n}{3}} \rp.
\end{align*}

\subsection{Estimates for $J_1^*(n)$ and $J_2^*(n)$}

Having bounded the explicit error terms, we now estimate the integrals $J^*_1(n), J_2^*(n)$ in terms of more familiar {\it $I$-Bessel functions}. Recall that for any real $s$, the function $I_s(x)$ may be defined by
\begin{align*}
	I_s(x) := \dfrac{1}{2\pi i} \int_{\tilde D} w^{-s-1} \exp\lp \dfrac{x}{2} \lp \dfrac{1}{w} + w \rp \rp dw,
\end{align*}
where $\tilde D$ is any contour that loops from $-\infty$ below $\R_{<0}$ around zero counterclockwise and back to $-\infty$ above $\R_{<0}$. Let $D_0 := \{ w \in \C : \mathrm{Re}\lp w \rp = x, \left| \mathrm{Im}\lp w \rp \right| \leq 15 x \}$ as earlier, and let
\begin{align*}
	D_{\pm} := \{ w \in \C : \mathrm{Re}\lp w \rp \leq x, \mathrm{Im}\lp w \rp = \pm 15 x \}.
\end{align*}
Define the (counterclockwise-oriented) path $D := D_- \cup D_0 \cup D_+$. Letting $\tilde D$ be the image of $D$ under the change of variables $z = \frac{\pi}{4\sqrt{3n}} w$, we can see that
\begin{align*}
	\dfrac{1}{2\pi i} \int_D z^{-\frac 14} \exp\lp \dfrac{\pi^2}{48z} + nz \rp dz = \dfrac{\pi^{\frac 34}}{2^{\frac 32} 3^{\frac 38} n^{\frac 38}} I_{-\frac 34} \lp \dfrac{\pi}{2} \sqrt{\dfrac{n}{3}} \rp
\end{align*}
and similarly
\begin{align*}
	\dfrac{1}{2\pi i} \int_D z^{\frac 14} \exp\lp \dfrac{\pi^2}{48z} + nz \rp dz = \dfrac{\pi^{\frac 54}}{2^{\frac 52} 3^{\frac 58} n^{\frac 58}} I_{-\frac 54}\lp \dfrac{\pi}{2} \sqrt{\dfrac{n}{3}} \rp.
\end{align*}
By changing variables $q \mapsto z$, we have
\begin{align*}
	J_1^*(n) = \dfrac{\Gamma\lp \frac 14 \rp}{2^{\frac 34} \pi^{\frac 12}} \cdot \dfrac{1}{2\pi i} \int_{D_0} z^{-\frac 14} \exp\lp \dfrac{\pi^2}{48z} + nz \rp dz,
\end{align*}
and therefore
\begin{align*}
	\dfrac{\Gamma\lp \frac 14 \rp \pi^{\frac 14}}{2^{\frac 94} 3^{\frac 38} n^{\frac 38}} I_{-\frac 34} \lp \dfrac{\pi}{2} \sqrt{\dfrac{n}{3}} \rp - J_1^*(n) = \dfrac{\Gamma\lp \frac 14 \rp}{2^{\frac 34} \pi^{\frac 12}} \cdot \dfrac{1}{2\pi i} \lp \int_{D_+} + \int_{D_-} \rp z^{-\frac 14} \exp\lp \dfrac{\pi^2}{48z} + nz \rp dz.
\end{align*}
The same procedure applied to $J_2^*(n)$ yields
\begin{align*}
	(-1)^n \dfrac{\Gamma\lp \frac 34 \rp \pi^{\frac 34}}{2^{\frac{11}{4}} 3^{\frac 58} n^{\frac 58}} I_{-\frac 54}\lp \dfrac{\pi}{2} \sqrt{\dfrac{n}{3}} \rp - J_2^*(n) = \dfrac{\Gamma\lp \frac 34 \rp}{2^{\frac 14} \pi^{\frac 12}} \cdot \dfrac{1}{2\pi i} \lp \int_{D_+} + \int_{D_-} \rp z^{\frac 14} \exp\lp \dfrac{\pi^2}{48z} + nz \rp dz.
\end{align*}
For the remainder, we define
\begin{align*}
	M_1(n) := \dfrac{\Gamma\lp \frac 14 \rp \pi^{\frac 14}}{2^{\frac 94} 3^{\frac 38} n^{\frac 38}} I_{-\frac 34} \lp \dfrac{\pi}{2} \sqrt{\dfrac{n}{3}} \rp, \hspace{0.2in} M_2(n) := (-1)^n \dfrac{\Gamma\lp \frac 34 \rp \pi^{\frac 34}}{2^{\frac{11}{4}} 3^{\frac 58} n^{\frac 58}} I_{-\frac 54}\lp \dfrac{\pi}{2} \sqrt{\dfrac{n}{3}} \rp,
\end{align*}
which are the main terms of $J_1^*(n)$, $J_2^*(n)$. For $t \in D_-$, set $t = \lp x - u \rp - 15x i$ for $u \geq 0$. Since $\mathrm{Re}\lp \frac{\pi^2}{48t} \rp \leq \frac{\pi}{4} \sqrt{\frac{n}{3}}$ and $|t| \geq 15x$, we have
\begin{align*}
	\left| t^{- \frac 14} \exp\lp \dfrac{\pi^2}{48t} + nt \rp \right| \leq |t|^{- \frac 14} \exp\lp \dfrac{\pi}{4} \sqrt{\frac{n}{3}} + n\lp x - u \rp \rp \leq \dfrac{2^{\frac 12} 3^{\frac 18} n^{\frac 18}}{15^{\frac 14} \pi^{\frac 14}} \exp\lp \dfrac{\pi}{2} \sqrt{\frac{n}{3}} - nu \rp.
\end{align*}
This bound holds not only for $t \in D_-$, but also $t \in D_+$, and therefore since $\int_0^\infty \exp\lp -nu \rp du = \frac{1}{n}$, we have
\begin{align*}
	\left| \dfrac{1}{2\pi i} \lp \int_{D_-} + \int_{D_+} \rp z^{-\frac 14} \exp\lp \dfrac{\pi^2}{48z} + nz \rp dz \right| &\leq 2 \lp \dfrac{2^{\frac 12} 3^{\frac 18} n^{\frac 18}}{15^{\frac 14} \pi^{\frac 14}} \exp\lp \dfrac{\pi}{2} \sqrt{\frac{n}{3}} \rp \rp \int_0^\infty e^{-nu} du \\ &= \dfrac{2^{\frac 32} 3^{\frac 18}}{15^{\frac 14} \pi^{\frac 14} n^{\frac 78}} \exp\lp \dfrac{\pi}{2} \sqrt{\frac{n}{3}} \rp.
\end{align*}
It therefore follows that
\begin{align*}
	\left| M_1(n) - J_1^*(n) \right| \leq \dfrac{2^{\frac 34} \Gamma\lp \frac 14 \rp}{3^{\frac 18} 5^{\frac 14} \pi^{\frac 34} n^{\frac 78}} \exp\lp \dfrac{\pi}{2} \sqrt{\frac{n}{3}} \rp < \dfrac{8}{5n^{\frac 78}} \exp\lp \dfrac{\pi}{2} \sqrt{\frac{n}{3}} \rp.
\end{align*}
Similarly, for $t \in D_\pm$, set $t = \lp x - u \rp \pm 15 \eta i$ for $u \geq 0$. Then $\left| t \right|^2 \leq 226 x^2 + u^2$. Since we have assumed $n > 4800$, it is clear that $226 x^2 + u^2 \leq 1 + u^2$, and so
\begin{align*}
	\left| t^{\frac 14} \exp\lp \dfrac{\pi^2}{48t} + nt \rp \right| \leq |t|^{\frac 14} \exp\lp \dfrac{\pi}{4} \sqrt{\frac{n}{3}} + n\lp x - u \rp \rp \leq \lp 1 + u^2 \rp^{\frac 18} \exp\lp \dfrac{\pi}{2} \sqrt{\dfrac{n}{3}} - nu \rp,
\end{align*}
and therefore
\begin{align*}
	\left| \dfrac{1}{2\pi i} \lp \int_{D_+} + \int_{D_-} \rp z^{\frac 14} \exp\lp \dfrac{\pi^2}{48z} + nz \rp dz \right| \leq 2 \exp\lp \dfrac{\pi}{2} \sqrt{\dfrac{n}{3}} \rp \int_0^\infty \lp 1 + u^2 \rp^{\frac 18} \exp\lp -nu \rp.
\end{align*}
Since $\lp 1 + u^2 \rp^{\frac 18} \leq 1 + u^{\frac 14}$ for $u > 0$, we have for $n > 1$ that
\begin{align*}
	\int_0^\infty \lp 1 + u^2 \rp^{\frac 18} \exp\lp -nu \rp \leq \int_0^\infty \exp\lp -nu \rp du + \int_0^\infty u^{1/4} \exp\lp -nu \rp du < \dfrac{2}{n}.
\end{align*}
and so
\begin{align*}
	\left| \dfrac{1}{2\pi i} \lp \int_{D_+} + \int_{D_-} \rp z^{\frac 14} \exp\lp \dfrac{\pi^2}{48z} + nz \rp dz \right| < \dfrac{4}{n} \exp\lp \dfrac{\pi}{2} \sqrt{\dfrac{n}{3}} \rp.
\end{align*}
As a consequence, we have
\begin{align*}
	\left| M_2(n) - J_2^*(n) \right| \leq \dfrac{2^{\frac{11}{4}} \Gamma\lp \frac 34 \rp}{\pi^{\frac 12} n} \exp\lp \dfrac{\pi}{2} \sqrt{\dfrac{n}{3}} \rp < \dfrac{5}{n} \exp\lp \dfrac{\pi}{2} \sqrt{\dfrac{n}{3}} \rp.
\end{align*}
Combining all the estimates made thus far, we may conclude that
\begin{align} \label{Effective Asymptotic}
	\left|a(n) - M_1(n) - M_2(n) \right| \leq E(n).
\end{align}
where
\begin{align} \label{Error Definition}
	E(n) := \dfrac{21\pi}{10} & \exp\lp \lp \dfrac{\pi}{4} + \dfrac{12}{5\pi} \rp \sqrt{\dfrac{n}{3}} \rp \notag \\ &+ \left[ \dfrac{4}{n^{\frac 78}} + \dfrac{5}{n} + \dfrac{13}{20 n^{\frac 98}} + \dfrac{21291081}{n^{\frac{19}{8}}} + \dfrac{17716899}{n^{\frac{21}{8}}} \right] \exp\lp \dfrac{\pi}{2} \sqrt{\dfrac{n}{3}} \rp.
\end{align}
Taken together, \eqref{Effective Asymptotic} and \eqref{Error Definition} now imply Theorem \ref{a(n) Asymptotics}.

\subsection{Analysis of Effective Inequalities}

In this section, we prove that $a(n) \geq 0$ for all $n \geq 0$. Note that in order to prove $a(n) \geq 0$ for a particular value of $n$, it would suffice to show that $a(n) \geq M_1(n) + M_2(n) - E(n)$. The majority of this proof consists in simplifying this sufficient condition on $n$ until an explicitly lower bound is achieved. 

For simplicity, it is easiest to remove the $(-1)^n$ from $M_2(n)$ by leveraging $M_2(n) \leq \left|M_2(n)\right|$. Thus, to prove $a(n) \geq 0$ it would suffice to prove that $M_1(n) - \left|M_2(n)\right| - E(n) \geq 0$, that is,
\begin{align} \label{First Reduced Form}
	M_1(n) \geq \left|M_2(n)\right| + E(n).
\end{align}
Note that to prove \eqref{First Reduced Form}, it would suffice to prove $M_1(n) \geq 2 \left|M_2(n)\right|$ and $M_1(n) \geq 2 E(n)$. We now prove these inequalities one at a time. Taking the definitions of $M_1(n)$ and $M_2(n)$, the inequality $M_1(n) \geq 2\left|M_2(n)\right|$ may be rearranged to the form
\begin{align*}
	\dfrac{I_{-\frac 34}\lp \dfrac{\pi}{2} \sqrt{\dfrac{n}{3}} \rp}{I_{-\frac 54}\lp \dfrac{\pi}{2} \sqrt{\dfrac{n}{3}} \rp} > \dfrac{2 \Gamma\lp \frac 34 \rp}{\Gamma\lp \frac 14 \rp} \sqrt{\dfrac{\pi}{6n}}
\end{align*}
Now, the $I$-Bessel function has the power series expansion
\begin{align*}
	I_s(t) = \lp \dfrac{t}{2} \rp^s \sum_{k\geq 0} \dfrac{t^{2k}}{4^k k! \Gamma\lp s + k + 1 \rp}
\end{align*}
from which one may clearly see that $I_{-\frac 34}(t) > I_{-\frac 54}(t)$ for all $t > 1$. In particular, it is clear that for all $n > 4800$ that $M_1(n) > 2\left|M_2(n)\right|$. 

For a fixed $n > 4800$, in order to prove $a(n) \geq 0$ we have shown that it will suffice to prove $M_1(n) > 2 E(n)$. We prove this result by first bounding $M_1(n)$ from below. By \cite[Exercise 13.2, pg. 269]{Olver}, we have for $t>0$ real that
\begin{align*}
	I_{-\frac 34}(t) = \dfrac{e^t}{\sqrt{2\pi t}}\lp 1 + \delta_1(t) \rp - i e^{-\frac 34 \pi i} \dfrac{e^{-t}}{\sqrt{2\pi t}} \lp 1 + \gamma_1(t) \rp,
\end{align*}
where $\delta_1(t), \gamma_1(t)$ satisfy the bounds
\begin{align*}
	\left| \gamma_1(t) \right| < \dfrac{5}{16t} \exp\lp \dfrac{5}{16t} \rp \hspace{0.1in} \text{and} \hspace{0.1in} \left| \delta_1(t) \right| < \dfrac{5\pi}{16t} \exp\lp \dfrac{5\pi}{16t} \rp.
\end{align*}
Therefore, we have
\begin{align*}
	\left| I_{-\frac 34}(t) - \dfrac{e^t}{\sqrt{2\pi t}} \right| < \dfrac{5\pi}{16 \sqrt{2\pi} t^{\frac 32}} \exp\lp t + \dfrac{5\pi}{16t} \rp + \dfrac{e^{-t}}{\sqrt{2\pi t}}\lp 1 + \dfrac{5}{16t} \exp\lp \dfrac{5}{16t} \rp \rp,
\end{align*}
from which it follows that
\begin{align*}
	I_{-\frac 34}(t) > \dfrac{e^t}{\sqrt{2\pi t}} - \left[ \dfrac{5\pi}{16 \sqrt{2\pi} t^{\frac 32}} \exp\lp t + \dfrac{5\pi}{16t} \rp + \dfrac{e^{-t}}{\sqrt{2\pi t}}\lp 1 + \dfrac{5}{16t} \exp\lp \dfrac{5}{16t} \rp \rp \right].
\end{align*}
We wish now to show $I_{-\frac 34}(t) > \frac{99 e^t}{100 \sqrt{2\pi t}}$ for suitably large $t$, for which it will suffice to consider their ratio (since both are positive). We have from the above inequality that $I_{-\frac 34}(t) \lp \frac{99e^t}{10\sqrt{2\pi t}} \rp^{-1}$ is plainly an increasing function of $t$, and so we can see that if we set $t = \frac{\pi}{2} \sqrt{\frac{n}{3}}$, the inequality holds for all $n > 4800$. Thus, to prove \eqref{First Reduced Form} for any given $n > 4800$ it will suffice to show that
\begin{align*} 
	\dfrac{99}{100} \cdot \dfrac{\Gamma\lp \frac 14 \rp}{2^{\frac 54} 3^{\frac 18} \pi^{\frac 34} n^{\frac 58}} \exp\lp \dfrac{\pi}{2} \sqrt{\dfrac{n}{3}} \rp &\geq \dfrac{21\pi}{5} \exp\lp \lp \dfrac{\pi}{4\sqrt{3}} + \dfrac{4\sqrt{3}}{5\pi} \rp \sqrt{n} \rp \\ &+ \left[ \dfrac{8}{n^{\frac 78}} + \dfrac{10}{n} + \dfrac{13}{10 n^{\frac 98}} + \dfrac{42582162}{n^{\frac{19}{8}}} + \dfrac{35433798}{n^{\frac{21}{8}}} \right] \exp\lp \dfrac{\pi}{2} \sqrt{\dfrac{n}{3}} \rp,
\end{align*}
which on dividing through by $\frac{1}{n^{5/8}} \exp\lp \frac{\pi}{2} \sqrt{\frac{n}{3}} \rp$ and making a convenient numerical estimate, it will suffice to show
\begin{align} \label{Final Reduced Form}
	\dfrac{21\pi n^{\frac 58}}{5} \exp\lp \lp \dfrac{12}{5\pi} - \dfrac{\pi}{4} \rp \sqrt{\dfrac{n}{3}} \rp + \left[ \dfrac{8}{n^{\frac 14}} + \dfrac{10}{n^{\frac 38}} + \dfrac{13}{10 n^{\frac 12}} + \dfrac{42582162}{n^{\frac{7}{4}}} + \dfrac{35433798}{n^2} \right] < \dfrac{11}{20}.
\end{align}
It is clear that for $n \geq 350000$ (in fact, much smaller $n$ will do) the left-hand side is a decreasing function of $n$. It can also be checked with a direct calculation that \eqref{Final Reduced Form} is true for $n = 350000$. Our method only assumed $n > 4800$, so we have now proven that $a(n) \geq 0$ for all $n \geq 350000$. The author has checked the values of $a(n)$ for $1 \leq n \leq 350000$ using his personal computer and found all to be non-negative. Therefore, Theorem \ref{MAIN} follows.

\end{document}